\documentclass[10pt]{article}

\usepackage[margin=1.5in]{geometry}  
\usepackage{graphicx}              
\usepackage{amsmath}               
\usepackage{amsfonts}             
\usepackage{amsthm}            
\usepackage{enumitem}
\usepackage{amssymb}
\usepackage{hyperref}
\usepackage{enumitem}

\usepackage{tikz}
\usetikzlibrary{matrix}

\tikzset{diagram/.style={matrix of math nodes, inner sep=0pt, row
    sep=#1, column sep=2.5em, text height=1.5ex, text depth=.25ex,
    nodes={inner sep=1ex}}}
\tikzset{diagram/.default=2.5em}
\newcommand\diagram{\path node[diagram]}

\newtheorem{thm}{Theorem}
\newtheorem{prop}[thm]{Proposition}
\newtheorem{lemma}[thm]{Lemma}
\newtheorem{cor}[thm]{Corollary}

\theoremstyle{definition}
\newtheorem{defn}{Definition}

\newtheorem{rmk}{Remark}

\newtheorem*{related works}{Related Works}

\theoremstyle{definition}
\newtheorem{question}{Question}

\numberwithin{equation}{section}

\newcommand{\RR}{\CC[t,t^{-1}]} 
\newcommand{\RRR}{\QQ[t,t^{-1}]} 
\newcommand{\CC}{\mathbb{C}}
\newcommand{\BB}{\mathbb{B}_n} 
\newcommand{\QQ}{\mathbb{Q}}

\newcommand{\ZZ}{\mathbb{Z}}
\newcommand{\ZZZ}{\ZZ[t,t^{-1}]} 

\newcommand{\D}{\mathcal{G}_n}

\newcommand{\s}{\sigma} 
\newcommand{\e}{\epsilon} 
\newcommand{\g}{\gamma} 
\newcommand{\PP}{\conf(\fq)} 
\newcommand{\A}{\mathbb{A}} 
\newcommand{\fq}{\mathbb{F}_q}
\newcommand{\ff}{\mathbb{F}}
\newcommand{\cn}{\text{Conf}_{n}(\CC)} 
\newcommand{\nn}{\Pn} 
\newcommand{\conf}{\text{Conf}_{n}} 
 
\newcommand{\Pn}{\mathrm{Conf}_n}

\AtEndDocument{\bigskip{\footnotesize%
\textsc{Department of Mathematics,}\par 
\textsc{University of Chicago,}\par
\textsc{5734 S. University Ave.}\par
\textsc{Chicago, IL 60637, U.S.A.} \par 
  \addvspace{\medskipamount} 
  \text{E-mail}: 
\texttt{chen@math.uchicago.edu} \par
}}

\begin{document}

\nocite{*}

\title{Homology of braid groups, the Burau representation, and points on superelliptic curves over finite fields}

\author{Weiyan Chen}

\date{\today}

\maketitle

\begin{abstract}
The (reduced) Burau representation $V_n$ of the braid group $B_n$ is obtained from the action of $B_n$ on the homology of an infinite cyclic cover of the disc with $n$ punctures. In this paper, we calculate $H_*(B_n;V_n)$. 
Our topological calculation has the following arithmetic interpretation (which also has different algebraic proofs): the expected number of points on a random superelliptic curve of a fixed genus over $\fq$ is $q$.
\end{abstract}


\section{Introduction}

Every configuration of $n$ distinct points $P=\{p_1,\cdots,p_n\}$ in $\CC$ determines a smooth \emph{superelliptic curve} over $\CC$:
\begin{equation}
\label{curve equation}
X_{P}\ :\ y^d=(x-p_1)\cdots (x-p_n)
\end{equation}
This construction has been studied by many people. See for example McMullen \cite{McMullen}. 
To study how the construction (\ref{curve equation}) varies in family, we compute the homology of braid groups with coefficients in the (complexified and reduced) Burau representation in Theorem \ref{braid} below.  Our result has an arithmetic interpretation via  \'etale cohomology theory. Namely, it implies that the expected number of points on a random superelliptic curve of a fixed genus over $\fq$ is $q$. 

\paragraph{The topological problem.}

Consider the variety $\conf:=\{(a_1,\cdots,a_n)\in \A^n\ :\ \Delta\ne0\}$ where $\Delta=\Delta(a_1,\cdots,a_n)$ stands for the discriminant of the polynomial $x^n+a_1x^{n-1}+\cdots+a_{n-1}x+a_n$. $\conf$ is the variety of monic, square-free polynomials of degree $n$. Its $\CC$-points, denoted by $\conf(\CC)$, parametrize configurations of $n$ distinct unlabeled points on $\CC$, because each polynomial in $\conf(\CC)$ gives $n$ distinct roots.  The affine curve $X_P$ in (\ref{curve equation}) associated to any configuration $P\in\cn$ gives rise to a fiber bundle over $\cn$:
\begin{equation}
\label{topo bundle}
\begin{tikzpicture}
\diagram (m)
{X_{P} & \mathrm{E}_{n,d}(\CC)\\
  \ &\cn \\};
\path [->] (m-1-1) edge node [above] {} (m-1-2)
           (m-1-2) edge node [right] {$\pi$} (m-2-2);
\end{tikzpicture}
\end{equation}
where $\mathrm{E}_{n,d}$ is the total space and $\pi$ is the bundle projection. The topological problem is to understand the map on cohomology groups induced by the fibration. To understand the fibration, we should study the monodromy representation. By the work of Fadell-Neuwirth \cite{FN}, $\cn$ is a $\mathrm{K}(B_n,1)$ where $B_n$ is the braid group on $n$ strands. The monodromy representation
\begin{equation}
\label{monodromy Burau}
B_n\cong \pi_1(\cn,P)\longrightarrow \mathrm{Aut}(H^1(X_P;\CC))
\end{equation}
decomposes into a direct sum of subrepresentations, each of which can be obtained from the \emph{reduced Burau representation}
\begin{equation}
\label{rho}
\rho: B_n\longrightarrow GL_{n-1}\ZZZ.
\end{equation}
by replacing the formal variable $t$ by a $d$-th root of unity. The precise definition of $\rho$ and its connection with the monodromy (\ref{monodromy Burau}) will be given in Section \ref{topo}. Let the \emph{reduced Burau module}  be $V_n:=\RR^{n-1}$ on which $B_n$ acts by $\rho$ in (\ref{rho}). The study of the fibration (\ref{topo bundle})  thus motivates us to calculate $H_*(B_n;V_n)$, which we answer by the following theorem:
\begin{thm}
\label{braid}
If  $n>2$, then
\[
H_k(B_n;V_n)\cong
  \begin{cases}
   0 \ \ \ &  k=0 \\
   \RR/(1-t) \ \ \ &  0<k<n-2 \\
   \RR/(1-t) \ \ \ & k=n-2,\text{and $n$ is odd} \\
   \RR/(1-t^2) \ \ \ & k=n-2,\text{and $n$ is even} \\
    0\ \ \ &k\geq n-1
  \end{cases}
\]
\end{thm}
If $n=2$, then we have $B_2\cong\ZZ$ and $V_2\cong \RR$. The generator of $B_2\cong\ZZ$ acts on $\RR$ via multiplication by $-t$. A standard computation shows that
$H_0(B_2;V_2) = \RR/(1+t)$ and $H_i(B_2;V_2) =0$ for all $i>0$. 

Theorem \ref{braid} will be proved  in Section \ref{topo} using the work of  Callegaro-Moroni-Salvetti \cite {CMS}. As a consequence of Theorem \ref{braid}, the fibration $\pi$ (\ref{topo bundle}) induces an isomorphism on all cohomology groups when either $d$ or $n$ is odd; see Proposition \ref{topo input} below.

\begin{rmk}[\textbf{Homological stability}]
From Theorem \ref{braid}, we see that for each fixed $k$, the homology group $H_k(B_n;V_n)$ will eventually be the same for all $n$ sufficiently large. Homological stability of braid groups with coefficients in the Burau representation  was proved independently by Wahl and Randal-Williams \cite{Wahl} as a consequence of a more general theorem. Theorem \ref{braid} concurs with their result, and indeed computes all the homology groups explicitly. 

The homological stability of $H_i(B_n;V_n)$ parallels a result of Church-Farb \cite{CF}. They proved that $H_k(B_n;U_n)$ stabilizes as $k$ is fixed and $n\to \infty$ when $\{U_n\}_{n=1}^\infty$ is any ``consistent" sequence of representations of $S_n$  viewed as representations of $B_n$ via the map $B_n\to S_n$ (see \cite{CF} for what ``consistent" means). The stabilization of $H_k(B_n;V_n)$ supplies a different example where the homological stability  holds  when the twisted coefficients do not factor through the quotient $B_n\to S_n$. 
\end{rmk}

The general problem of computing the homology of the braid groups $B_n$ with twisted coefficients is a topic of independent interest and has been considered by many authors. For example, $H_*(B_n;W)$ has been computed for:

\begin{enumerate}[label=(\roman*)]
\item (Cohen, \cite{Cohen}; Vassiliev \cite{V}): $W=\CC$ where $B_n$ acts on $\CC$ by the composition $B_n\to S_n\xrightarrow{\text{sign}}\{\pm1\}\subset\CC^\times$.  

\item (Cohen, \cite{Cohen}; Vassiliev \cite{V}): $W = \CC^{n-1}$ where  $B_n$ acts by composing $B_n\to S_n$ with the standard representation $S_n\curvearrowright \CC^{n-1} = \{(x_1,\cdots,x_n)\in \QQ^n\ |\ \sum x_i=0\}$.

\item (De Concini-Procesi-Salvetti \cite{CPS}): 
$W = \RR$ where $B_n$ acts by the determinant of the reduced Burau representation. 
\end{enumerate}
These three representations of $B_n$ are related to the reduced Burau representation $\rho$ by the following commutative diagram
\begin{equation*}
\begin{tikzpicture}
\diagram (m)
{ B_n &&& GL_{n-1}\RR & & & GL_{n-1}\QQ\\
  & & &\RR^\times & & & \QQ^\times\\};
\path [->] (m-1-1) edge node [above] {$\rho$} (m-1-4)
           (m-1-4) edge node [left] {det} (m-2-4)
           (m-1-4) edge node [above] {$t\mapsto1$} (m-1-7)
           (m-1-7) edge node [right] {det} (m-2-7)
           (m-2-4) edge node [above] {$t\mapsto1$} (m-2-7);
\end{tikzpicture}
\end{equation*}
Moreover, it is possible to recover (ii) for $H_*(B_n;\CC^{n-1})$ from Theorem \ref{braid} using the Universal Coefficient Theorem by setting $t=1$. This is explained in Remark \ref{Coxeter}.

\paragraph{Connection to arithmetic.}
We consider the construction (\ref{curve equation}) over a finite field $\fq$, where $q$ is a power of prime number. Every monic polynomial $f(x)\in\fq[x]$ determines an affine curve over $\fq$ 
$$X_f\ :\ y^d=f(x).$$
We will only consider when $X_f$ is a smooth curve, or equivalently, when $\gcd(q,d)=1$ and $f(x)$ is square-free, \emph{i.e.} has distinct roots in $\overline{\fq}$. In this case, $X_f$ is a smooth {superelliptic curve} over $\fq$. The $\fq$-points on $X_f$
$$X_f(\fq) = \{(x,y)\in\fq^2\ :\ y^d=f(x)\}$$ is a finite set. It is natural to ask what the average of $|X_f(\fq)|$ is as the curve $X_f$ varies among a natural family. Observe that the family of smooth superelliptic curves of a fixed genus is parametrized by the following finite set:
$$\conf(\fq):= \{f\in\fq[x]\ :\ \text{$f$ is monic, square-free, and of degree $n$}\}.$$
which contains precisely the $\fq$-points of the variety $\conf$.  
The following is an arithmetic interpretation of Theorem \ref{braid} via the \'etale cohomology theory.

\begin{cor}
\label{average pt}
Suppose that either $n$ or $d$ is an odd number. For all but finitely many primes $p$ (depending on $d$ and $n$),  for every $q$ a power of $p$, we have
\begin{equation}
\label{exp}
\frac{1}{|\conf(\fq)|}\sum_{f\in\PP} |X_f(\fq)| = q
\end{equation}
\end{cor}

\begin{rmk}
Instead of affine curves, one can also consider their smooth completions and study the expected number of $\fq$-points on the smooth completions of $X_f$. 
The number of points at infinity on the smooth completion of $X_f(\fq)$ is equal to the number of $e$-th roots of unity in $\fq$ where $e=\gcd(d,n)$. 
Therefore, in this case, knowing the distribution of number of $\fq$-points on the smooth completions is equivalent to knowing it on their affine models.
\end{rmk}

The asymptotic statistics on families of varieties over $\fq$ is an active area of research. See David's notes \cite{David} for a survey. For instance, the asymptotic distribution (as genus $\to\infty$) of the number of $\fq$-points on curves determined by $y^d=f(x)$ has been studied by Kurlberg-Rudnick \cite{KR} in the case when $d=2$ (hyperelliptic curves), by Bucur-David-Feigon-Lal\'in \cite{BDFL} in the case when $d=3$ (trigonal curves), and by Xiong \cite{Xiong} and Cheong-Wood-Zaman \cite{CWZ} for general $d$ relatively prime to $q$ (superelliptic curves). In comparison to the previous work, Corollary \ref{average pt} is only concerned about the expected value as opposed to the exact distribution. On the other hand, Corollary \ref{average pt} applies to finite $n$ instead of applying only at the limit $n\to\infty$. 
The topological approach in this paper is different from the methods used in all the aforementioned papers. For a discussion on the possibility of adapting the topological method to calculate the higher moments of statistics such as the variance, please see Section \ref{higher moments}.

Even though in this paper Corollary \ref{average pt} will be proved using Theorem \ref{braid}, we remark that our topological method is not the only way to prove the arithmetic result. After Corollary \ref{average pt} was announced, Will Sawin, Frank Calegari and Bjorn Poonen [personal communications] independently found proofs of the formula (\ref{exp}) using purely algebraic method. The main point here is to explain how our topological calculation can give correct predictions and conceptual explanations for the arithmetic result.

\begin{rmk}[The necessity of hypothesis in Corollary \ref{average pt}]
%

The condition on excluding finitely many primes is not necessary. This condition is needed only in a technical step (Lemma \ref{primes}) of our proof, and can be removed via other approach. The condition that either $n$ or $d$ is odd is necessary.  In fact, the topological computation will predict an error term in equation (\ref{exp}) if both $n$ and $d$ are even. See Section \ref{even} for more discussion in this case. 
\end{rmk}

The connection between Theorem \ref{braid} (topology) and Corollary \ref{average pt} (arithmetic) is given by the following algebro-geometric consideration. 
The fiber bundle  $E_{n,d}(\CC)\to\cn$ (\ref{topo bundle}) can be taken in the algebraic category. 
 More precisely, there is a morphism $\pi:E_{n,d}\to\conf$ of algebraic varieties defined over $\ZZ$ such that the fiber bundle is just $\pi$ restricted to their complex points.  Corollary \ref{average pt} is equivalent to the following statement about their $\fq$-points:
\begin{equation}
\label{ratio}
\frac{|\mathrm{E}_{n,d}(\fq)|}{|\conf(\fq)|}=q. 
\end{equation}
To compute the ratio, we use the Grothendieck-Lefschetz fixed point formula in \'etale cohomology, which gives (after applying Poincar\'e duality):
\begin{align}
\label{GLE1}
|\mathrm{E}_{n,d}(\fq)| =q^{\dim(\mathrm{E}_{n,d})}\sum_{i\geq0} (-1)^i \mathrm{trace}(\mathrm{Frob}_q^*: H^i_{\text{\'et}}({\mathrm{E}_{n,d}}_{/\overline{\mathbb{F}}_q};\QQ_{l})^\vee)\\
\label{GLE2}|\nn(\fq)| =q^{\dim(\nn)}\sum_{i\geq0} (-1)^i \mathrm{trace}(\mathrm{Frob}_q^*: H^i_{\text{\'et}}({\nn}_{/\overline{\mathbb{F}}_q};\QQ_{l})^\vee)
\end{align}
where $\mathrm{Frob}_q$ denotes the Frobenius map. 

To prove (\ref{ratio}), our input from topology will be the following consequence of Theorem \ref{braid}: the fiber bundle map $\pi:E_{n,d}(\CC)\to\cn$ induces isomorphisms on all singular cohomology groups when either $d$ or $n$ is odd (Proposition \ref{topo input}). This implies, via general theorems in \'etale cohomology (precisely, the Artin comparison theorem, base change for compactly supported cohomology, and the constructibility of the derived pushforward, see Section \ref{etale} for details), that $\pi$ as a map of varieties induces an isomorphisms on all the \'etale cohomology groups:
$$H^i_{\text{\'et}}({\mathrm{E}_{n,d}}_{/\overline{\mathbb{F}}_q};\QQ_{l})\cong H^i_{\text{\'et}}({\conf}_{/\overline{\mathbb{F}}_q};\QQ_{l})$$
and the isomorphisms are equivariant with respect to $\mathrm{Frob}_q$. 
Taking the quotient of (\ref{GLE1}) and (\ref{GLE2}), we conclude
$$\frac{|E_{n,d}(\fq)|}{|\conf(\fq)|}=\frac{q^{\dim E_{n,d}}}{q^{\dim \conf}}=\frac{q^{n+1}}{q^n}=q
$$
hence establishing Corollary \ref{average pt}. 

This paper is very much inspired by the work of Church-Ellenberg-Farb \cite{CEF}, where they related the cohomology of the braid groups with coefficients in various representations of the symmetric groups to statistics on the roots of monic square-free polynomials, which can be viewed as smooth 0-dimensional varieties. The present paper relates the homology of the braid groups in the reduced Burau representation to statistics on families of smooth 1-dimensional varieties, namely, curves. Furthermore, the results of this paper  supply an example of a broader program applying topology to make concrete calculations for various moduli spaces over finite fields, using \'etale cohomology. See \cite {FW1} and \cite{FW2} for more examples of this program.\\

\textbf{Acknowledgments.}\ \ \ \
The author is grateful to Sasha Beilinson, Joan Birman, Fred Cohen, Frank Calegari, Jordan Ellenberg, Matt Emerton, Sean Howe, Madhav Nori, Bjorn Poonen, and Will Sawin for helpful conversations on the subject. In particular, Section \ref{counting} of this paper is inspired by conversations with Jordan Ellenberg. The author would like to thank Sean Howe, Curt McMullen,  Bjorn Poonen, and the anonymous referee for suggestions on early drafts of this paper. Finally, the author is pleased to express his deep gratitude to his advisor Benson Farb, both for his continued support on the project and for his extensive comments on the exposition.

\section{Braid groups and reduced Burau representation}
\label{topo}

In this section, we prove Theorem \ref{braid}, based on Callegaro-Moroni-Salvetti's calculation of homology of type-B Artin groups. 

\subsection{Topological interpretation of reduced Burau representation}

We first recall the definitions of the braid groups and the reduced Burau representation, emphasizing their topological interpretation. All of the results in this subsection are classical; proofs can be found in Chapter 1 and Chapter 3 of \cite{KT}.

The \emph{braid group} $B_n$ on $n$ strands has the following presentation 
\begin{align*}
\langle \s_1,\cdots,\s_{n-1}\ |\ &\s_i\s_{i+1}\s_i=\s_{i+1}\s_i\s_{i+1}& \text{for }i=1,\cdots,n-2\\
&\s_i\s_j=\s_j\s_i &\text{for }|i-j|>1\rangle
\end{align*}
The \emph{reduced Burau representation} has been classically defined  as  the group homomorphism 
$$\rho:B_n\to GL_{n-1}(\ZZZ)$$
which in the case $n>2$ sends the generator $\s_1,\cdots,\s_{n-1}$ to the following matrices
$$\rho(\s_1)=
\left( \begin{array}{cc|c}
 - t & 0 & \ \\
1 & 1 & \ \\
\hline
\ & \ & I_{n-3} \end{array} \right), 
\ \ \ \ \rho(\s_{n-1})=
\left( \begin{array}{c|cc}
 I_{n-3}& \ & \ \\
\hline
\ & 1 & t \\
\ & 0 & -t \end{array} \right),$$
and for $1<i<n-1$, 
$$
\rho(\s_i)=
\left( \begin{array}{c|ccc|c}
I_{i-2}&\ &\ &\ &\ \\
\hline
\ &1&t&0&\ \\
\ &0&-t&0&\ \\
\ &0&1&1&\ \\
\hline
\ &\ &\ &\ &I_{n-i-2}
\end{array} \right)$$
where $I_m$ stands for the $m\times m$ identity matrix. When $n=2$, we have $B_2=\langle \s\rangle\cong\ZZ$. The reduced Burau representation sends the generator $\s$ to $-t\in \ZZZ^\times= GL_1(\ZZZ)$.

There is a topological interpretation of the reduced Burau representation, which will be important for our purposes. Let $D_n$ denote the closed disc with $n$ punctures in the interior. There is an isomorphism $B_n\cong \mathrm{Mod}(D_n)$, where $\mathrm{Mod}(D_n)$ consists of homeomorphisms of $D_n$ fixing $\partial D_n$ pointwise up to isotopies fixing $\partial D_n$ pointwise. By choosing a basepoint $p\in\partial D_n$, the braid group
$B_n\cong \mathrm{Mod}(D_n)$ acts on $\pi_1(D_n,p)$. Let $\g_i\in\pi_1(D_n,p)$ be the element represented by a loop that winds once  around the $i$-th puncture counterclockwise, then $\pi_1(D_n,p)\cong F_n$, a free group of rank $n$ generated by $\g_1,\cdots,\g_n$. The action $B_n\curvearrowright \pi_1(D_n,p)$ has the following formula
\begin{equation}
\label{braid action}
\phi:B_n\longrightarrow \mathrm{Aut}(F_n)
\end{equation}
\begin{align*}
\phi(\s_i):\ \g_i&\mapsto \g_{i+1}\\
   \g_{i+1}&\mapsto \g_{i+1}^{-1}{\g_i}\g_{i+1}\\
   \g_k&\mapsto \g_k \ \ \  \text{for all }k\ne i,i+1 
\end{align*}
Geometrically, $\s_i$ represents a counterclockwise half-twist about the $i$-th and $(i+1)$-th punctures.

The total winding number $\omega$ of $\gamma\in \pi_1(D_n,p)$ is
\begin{equation}
\label{winding_def}
\omega(\gamma):= \sum_{i=1}^{n} \text{the winding number of $\gamma$
relative to the $i$-th puncture}.
\end{equation}
This defines a surjective homomorphism $\omega: \pi_1(D_n,p)\to \ZZ$, with $\omega(\g_i)=1$ for all $i=1,\cdots,n$. 
From the formula (\ref{braid action}) we can see that the action of $B_n$ on $\pi_1(D_n,p)\cong F_n$ satisfies
\begin{equation}
\label{winding}
\omega(\sigma\cdot\gamma)=\omega(\gamma),\ \ \ \ \forall\s\in B_n,\forall\g\in F_n
\end{equation}
Let $X_n$ denote the infinite cyclic cover of $D_n$ corresponding to the kernel of the group homomorphism $\omega: \pi_1(D_n,p)\to \ZZ$. Because of (\ref{winding}), any mapping
class $\sigma$ of the punctured disc will lift to a mapping class on the cyclic cover $X_n$. In this way $B_n$ acts on $H_1(X_n;\ZZ)$. A standard computation shows that $H_1(X_n;\ZZ)$ is a free $\ZZZ$-module of rank $n-1$. In fact, the action of $B_n$ on $H_1(X_n;\ZZ)$ is equivalent to the reduced Burau representation in the sense that there is a commutative diagram as follows
\begin{equation}
\begin{tikzpicture}
\label{commute}
\diagram (m)
{B_n & GL_{n-1}\ZZZ\\
 \mathrm{Mod}(D_n) \ & \mathrm{Aut}(H_1(X_n;\ZZ)) \\};
\path [->] (m-1-1) edge node [above] {$\rho$} (m-1-2)
                   edge node [left] {$\cong$} (m-2-1)
           (m-1-2) edge node [right] {$\cong$} (m-2-2)
           (m-2-1) edge node [above] {} (m-2-2);
\end{tikzpicture}
\end{equation}
All arrows are group homomorphisms. For a proof that the diagram commutes, please see Theorem 3.7 and Remark 3.11 in \cite{KT}.

\subsection{Proof of Theorem \ref{braid}}

In this subsection, we prove Theorem \ref{braid}. The strategy is to use the Hochschild-Serre spectral sequence to relate $H_*(B_n;V_n)$ and the homology of  Artin groups of type $\mathbb{B}$ over a rank 1 representation $\RR$.

\begin{lemma}
\label{reduction}
There is an isomorphisms of $\RR$-modules
$$H_*(B_n;V_n)\cong H_*(B_n;H_1(F_n;\RR))$$
where on the right hand side, the action of $F_n$ on $\RR$ is given by $\g_i\mapsto t$, and the action of $B_n$ on $H_1(F_n;\RR)$ is induced by the action $B_n\curvearrowright F_n$ given in \ref{braid action}. 
\end{lemma}

\begin{proof}
Since (\ref{commute}) is a commutative diagram with the two vertical arrows being isomorphisms, we immediately have 
$$H_*(B_n;V_n)\cong H_*(B_n;H_1(X_n;\CC)).$$
On the other hand, the total winding number $\omega: F_n\to\ZZ$ gives the following isomorphisms
\begin{align}
H_1(X_n;\CC)\nonumber&\cong H_1(Ker\omega;\CC) &\text{$X_n$ is a $K(Ker\omega,1)$}\\
\nonumber& \cong H_1(F_n;Ind^{F_n}_{Ker\omega}\CC) \nonumber&\text{Shapiro's lemma}\\
\nonumber& \cong H_1(F_n;\CC[F_n/{Ker\omega}]) &\\
\nonumber& \cong H_1(F_n;\CC[\ZZ]) &\\
\label{iso}& \cong H_1(F_n;\RR)&\text{identifying $\ZZ\cong\langle t\rangle$ by $1\mapsto t$}
\end{align}
We check that the isomorphisms in (\ref{iso}) are all maps of $(B_n,\langle t\rangle)$-bimodules.
In $H_1(F_n;\RR)$, each generator $\g_i$ of $F_n$ acts on $\RR$ by multiplying by $t$. Moreover, $H_1(F_n;\RR)$ has an action of $B_n$ which is induced by the action $\phi:B_n\to \mathrm{Aut}(F_n)$ given in \ref{braid action}. Since the $B_n$-action on $H_1(X_n;\CC)$ is defined by lifting mapping classes to the cover $X_n$, it necessarily commutes with the action of the deck group $\ZZ\cong\langle t\rangle$.  Therefore the claim follows.
\end{proof}

The right hand side of Lemma \ref{reduction} suggests us to look at the semidirect product $F_n\rtimes_\phi B_n$ where $\phi:B_n\to \mathrm{Aut}(F_n)$ is given in (\ref{braid action}). It turns out that $F_n\rtimes_\phi B_n$ is isomorphic to the Artin group of type $\BB$, which we will explain below. 

\begin{defn}
The Artin group of type $\BB$, which we will denote by $\D$, is the group with the following presentation 
\begin{align}
\nonumber\D = \langle \e_1,\cdots,\e_{n-1},\e_n\ |\ &\e_i\e_j=\e_j\e_i\ \ \text{for } |i-j|>1\\
&\nonumber\e_i\e_{i+1}\e_i=\e_{i+1}\e_{i}\e_{i+1}\ \ \text{for }i< n-1\\
&\label{Artin's presentation} \e_{n-1}\e_n\e_{n-1}\e_n=\e_n\e_{n-1}\e_n\e_{n-1} \rangle.
\end{align}
\end{defn}

There is an isomorphism $\D\to F_n\rtimes_\phi B_n$.  The isomorphism can be given explicitly on the standard generators $\g_1,\cdots,\g_n$ of $F_n$ and $\s_1,\cdots,\s_{n-1}$ of $B_n$.
\begin{align}
F_n\rtimes_{\phi}B_n &\longrightarrow \D &\ \label{chow}\\
\nonumber \s_i&\mapsto \e_i  &\text{for $i=1,\cdots,n-1$}\\
\nonumber \g_i&\mapsto (\e_i\e_{i+1}\cdots \e_{n-1})\e_n(\e_i\e_{i+1}\cdots \e_{n-1})^{-1} &\text{for $i=1,\cdots,n-1$} \\
\nonumber\g_n&\mapsto \e_n &\
\end{align}
For a proof of the fact that (\ref{chow}) is an isomorphism, please refer to \cite{CP}.

\begin{rmk}
A geometric way to see $\D\cong F_n\rtimes_\phi B_n$ is to recognize that $\D$ and $F_n\rtimes_\phi B_n$ give different presentations of the same group, namely, the \emph{annular braid group} which is the subgroup of $B_{n+1}$ that leaves the $(n+1)$-th puncture invariant. The study of the annular braid group dates back to Chow in the 40's. We will not need this perspective in this paper and refer the reader to   \cite{Chow} and \cite{CMS} for more discussion on the relation among these groups. 
\end{rmk}

\begin{proof}[Proof of Theorem \ref{braid}.]

First of all, our proof  will use Theorem 4.2 of \cite{CMS} in a crucial way. 
\begin{thm}[Callegaro-Moroni-Salvetti]
\label{abelian}
Suppose $\RRR$ is the $\D$-module where $\e_n$ acts by $t$-multiplication and $\e_1,\cdots,\e_{n-1}$ acts trivially, then
\[
H^k(\D;\RRR)=
  \begin{cases}
   \RRR/(1-t) \ \ \ &  k=1,\cdots,n-1 \\
   \RRR/(1-t) \ \ \ & k=n,\text{and $n$ is odd} \\
   \RRR/(1-t^2) \ \ \ & k=n,\text{and $n$ is even} 
  \end{cases}
\]
\end{thm}

Using Theorem \ref{abelian}, we calculate the homology using the Universal Coefficient Theorem.

\begin{lemma}
\label{homology}
$H_{k-1}(\D;\RR)\cong H^k(\D;\RR)$ as $\RR$-modules.
\end{lemma}
\begin{proof}
 We will use $R$ to abbreviate $\RR$ in this proof only. 

Observe $R$ is a principal ideal domain. The Universal
Coefficient Theorem gives a short exact sequence
$$0\to Ext_{R}^1(H_{k-1}(\D;R),R)\to
H^k(\D;R)\to Hom_{R}(H_k(\D;R),R)\to0$$
Note that $H^k(\D;R)$ is always a torsion $R$-module by Theorem \ref{abelian}. Therefore, $Hom_{R}(H_k(\D;R);R)=0$ because it is both free and
a quotient of a torsion module. This implies that
$H_k(\D;R)$ contains no free submodule and hence is
torsion for all $k$. Therefore,
$$H_{k-1}(\D;R)\cong
Ext_{R}^1(H_{k-1}(\D;R),R)\cong
H^{k}(\D;R)).$$
\end{proof}

On the other hand, the isomorphism $\D\cong F_n\rtimes_\phi B_n$
 yields a split exact sequence 
$$1\to F_n\to \D\to B_n\to1$$
where $F_n$ is generated by $\g_1,\cdots, \g_n$, and $B_n$ is generated by $\s_1,\cdots,\s_{n-1}$. 
Applying the Hochschild-Serre spectral sequence to the short exact sequence yields a spectral sequence whose $E^2$ page is of the form $$E^2_{p,q}=H_p(B_n;H_q(F_n;\RR))\Rightarrow H_{p+q}(\D;\RR).$$ 
where each generator $\g_i$ of $F_n$ acts on $\RR$ by $t$-multiplication, and the action of $B_n$ on $H_q(F_n;\RR)$ is induced by $\phi:B_n\to \mathrm{Aut}(F_n)$ in \ref{braid action}.
The $E^2$ differential goes $d^2: E^2_{p,q}\to E^2_{p-2,q+1}$, hence is zero except when $p\geq 2$ and $q=0$. When $q=0$, $H_q(F_n;\RR) = \CC$ and $B_n$ acts trivially on it. When $p\geq 2$, $E^2_{p,0}=H_p(B_n;\CC)=0$. Therefore,  all differentials after the $E^2$ page are zero and $E^2=E^\infty$. Moreover,  $E^2_{p,q}$ are all torsion $\RR$-modules, which means the action of $\langle t \rangle\cong\ZZ$ factors through a finite cyclic group $\ZZ/m\ZZ$ for some $m$. By semisimplicity, there is no extension problem.
Therefore, we have
\begin{align}
\nonumber H_k(\D;\RR)&\cong \bigoplus_{p+q=k}E^2_{p,q} = E^2_{k,0}\oplus E^2_{k-1,1}\\
\label{spectral sequence}&=H_k(B_n;H_0(F_n;\RR))\oplus H_{k-1}(B_n;H_1(F_n;\RR))
\end{align}

Moreover, 
\begin{itemize}
\item $H_k(B_n;H_0(F_n;\RR)) = H_k(B_n;\CC)$, with trivial $\CC$-coefficients, because $H_0(F_n;\RR) = \RR_{F_n} = \RR/(1-t)$, on which $B_n$ acts trivially.
\item $H_{k-1}(B_n;H_1(F_n;\RR)) = H_{k-1}(B_n; V_n)$ by Lemma \ref{reduction}.
\end{itemize}

Finally, we combine all the isomorphisms of modules over $\RR$:
\begin{align*}
H_{k}(B_n;V_n)&\cong H_{k+1}(\D;\RR) /
H_{k+1}(B_n;\CC)&\text{by (\ref{spectral sequence})}\\
& \cong H^{k+2}(\D;\RR) /  H_{k+1}(B_n;\CC)\ \ \ \ \ &\text{by Lemma \ref{homology}}
\end{align*}
When $n>2$, Theorem \ref{abelian} says
\[
H^{k+2}(\D;\RR)=
  \begin{cases}
   \RR/(1-t) \ \ \ &  k=0,\cdots,n-3 \\
   \RR/(1-t) \ \ \ & k=n-2,\text{and $n$ is odd} \\
   \RR/(1-t^2) \ \ \ & k=n-2,\text{and $n$ is even} 
  \end{cases}
\]
Arnol'd's computation \cite{Arnol'd} of the homology of the braid group with the trivial coefficients yields 
\[
H_{k+1}(B_n;\CC)=
  \begin{cases}
   \CC \ \ \ &  k=0\\
   0 \ \ \ & k>0
  \end{cases}
\]
Taking the quotient of the two modules, we conclude
\[
H_k(B_n;V_n)=
  \begin{cases}
   0 \ \ \ & k=0\\
   \RR/(1-t) \ \ \  & 0< k< n-2 \\
   \RR/(1-t) \ \ \ & k=n-2,\text{and $n$ is odd} \\
   \RR/(1-t^2) \ \ \ & k=n-2,\text{and $n$ is even} 
  \end{cases}
\]
Theorem \ref{braid} is established.
\end{proof}

\subsection{The reduced Burau representation specialized at $t=\zeta$}
\label{specialized Burau section}

For any $\zeta\in\CC\setminus \{0\}$, we can obtain an $(n-1)$-dimensional representation of $B_n$ over $\CC$ if we replace the formal variable $t$ in the reduced Burau representation by $\zeta$. More precisely, we define the following $B_n$-representations
\begin{itemize}
\item $\CC_\zeta:=\CC[t,t^{-1}]/(t-\zeta)$ with trivial $B_n$ action.
\item $V_n(\zeta):= V_n\underset{\RR}\otimes \CC_{\zeta},$ where $V_n$ and $\CC_\zeta$ have $B_n$ actions as defined before.
\end{itemize}
In fact, $\CC_\zeta$ and $V_n(\zeta)$ are both $(B_n,\langle t\rangle)$-bimodules. This gives a family of $(n-1)$ dimensional representations $$\rho_\zeta:B_n\to GL(V_n(\zeta))\cong GL_{n-1}(\CC),\ \ \ \ \ \ \ \forall\zeta\in\CC\setminus\{0\}.$$ 
$V_n(1)$ is also called \emph{the reduced Coxeter representation}
;  $V_n(-1)$ is also called \emph{the reduced integral Burau representation}.

In order to compute $H_*(B_n;V_n(\zeta))$ from $H_*(B_n;V_n)$, we apply Universal Coefficient Theorems for twisted coefficients. See Th\'eor\`eme I.5.4.2. and Th\'eor\`eme I.5.5.2 of \cite{G}.
\begin{thm}[Universal Coefficient Theorem for twisted coefficients]
\label{UCT}
For $G$ a group and $R$ a principal ideal domain. Suppose $V$ is a $RG$-module, and $M$ is an  $R$-module with trivial $G$-action. There are the following split short exact sequences of $R$-modules:
\begin{enumerate}[label=(\roman*)]
\item $0\to H_k(G; V)\underset{R}\otimes M \to H_k(G; V\underset{R}\otimes M)\to Tor(H_{k-1}(G;V),M)\to 0$
\item $0\to Ext(H_{k-1}(G;V),M)\to H^k(G; Hom_R(V, M))\to Hom_R(H_k(G; V), M)\to 0.$
\end{enumerate}
\end{thm}
%
%
%


Let us denote
\begin{equation}
\label{check}
V_n(\zeta)^{\vee}:= Hom_\CC(V_n(\zeta),\CC)
\end{equation}
where $B_n$ acts on $V_n(\zeta)^{\vee}$ from the left by $(\s\cdot f)(v):=f(\s^{-1}\cdot v)$. The following proposition computes $H_*(B_n;V_n(\zeta))$ and $H^*(B_n;V_n(\zeta)^\vee)$ for all $\zeta\in\CC\setminus\{0\}$. 
\begin{prop}
\label{specialized Burau}
Suppose $n>2$ and $\zeta\in\CC\setminus\{0\}$, we have 
\begin{enumerate}
\item If $\zeta=1$, 
\[
dim_\CC H_k(B_n;V_n(\zeta))=
  \begin{cases}
   1 \ \ \ &  k=1 \\
   2 \ \ \ & 1<k<n-1 \\
   1\ \ \ &k= n-1\\
  \end{cases}
\]

\item If $\zeta=-1$ and $n$ is an even number
\[
dim_\CC H_k(B_n;V_n(\zeta))=
   1 \ \ \   k=n-2,n-1 
\]

\item If $\zeta^2\ne1$, or if $\zeta=-1$ and $n$ is an odd number
\[
H_k(B_n;V_n(\zeta))=0\ \ \ \ \text{for all $k$}
\]

\item For all $\zeta\in\CC\setminus\{0\}$, 
$$dim_\CC H^k(B_n;V_n(\zeta)^\vee) = dim_\CC H_k(B_n;V_n(\zeta))\ \ \ \ \forall k.$$

\end{enumerate}
\end{prop}

\begin{proof}
Claims (1)-(3) follow directly from Theorem \ref{braid} and Theorem \ref{UCT}, (i). Claim (4) follows from the Theorem \ref{UCT}, (ii).

%
%
%
%
\end{proof}
\begin{rmk}
\label{Coxeter}
When $\zeta=1$, Proposition \ref{specialized Burau} recovers Vassiliev's computation \cite{V} of $H^*(B_n;\CC^{n-1})$ where $B_n$ acts on $\CC^{n-1}$ by the reduced Coxeter representation. 
\end{rmk}

\section{The expected number of $\fq$-points on a random superelliptic curve}
\label{counting}

In this section, we will apply the topological results in the previous section to calculate the expected number of $\fq$-points on a random superelliptic curve. The bridge that connects topology and counting $\fq$-points is provided by the theory of \'etale cohomology and the Grothendieck-Lefschetz fixed point formula. The paper \cite{CEF} by Church-Ellenberg-Farb contains a very accessible explanation of this connection.

Fix natural numbers $d$ and $n$.  Consider the variety:
\begin{align*}
\mathrm{E}_{n,d} &= \{(x,y,a_1,\cdots,a_n)\in \A^{n+2}\ :\ \Delta\ne0, \ y^d = x^n+a_1x^{n-1}+\cdots+a_{n-1}x+a_n\}
\end{align*} 
where $\Delta=\Delta(a_1,\cdots,a_n)$ again stands for the discriminant of the polynomial $x^n+a_1x^{n-1}+\cdots+a_{n-1}x+a_n$. There is a morphism 
\begin{align}
\label{variety}
\pi: \mathrm{E}_{n,d}&\longrightarrow\Pn\\
\nonumber(x,y,a_1,\cdots,a_n)&\mapsto (a_1,\cdots,a_n)
\end{align}
Observe that the varieties $\mathrm{E}_{n,d}$ and $\Pn$, and the morphism $\pi$ are all defined by polynomial equations with coefficients in $\ZZ$. Therefore, we can consider the morphism $\pi$ on the $\CC$-points and on the $\fq$-points of the varieties, respectively. 

\subsection{The topology of the $\CC$-points}
\label{C-points}

In this subsection we will analyze the topological properties of the morphism $\pi$  on the $\CC$-points: 
$$\nonumber\pi_{/\CC}: \mathrm{E}_{n,d}(\CC)\longrightarrow \nn(\CC)$$
The map $\pi$ gives the following fiber bundle mentioned in the Introduction:
\begin{equation*}
\begin{tikzpicture}
\diagram (m)
{X_f(\CC) & \mathrm{E}_{n,d}(\CC) \\
  \ & f\in\nn(\CC)\\};
\path [->] (m-1-1) edge node [above] {} (m-1-2)
           (m-1-2) edge node [right] {$\pi_{/\CC}$} (m-2-2);
\end{tikzpicture}
\end{equation*}
The fiber is an affine superelliptic curve over $\CC$:
$$X_f(\CC):=\{(x,y)\in\CC^2\ :\ y^d=f(x)\}.$$
For simplicity of notation, we will use $\pi$ to denote $\pi_{/\CC}$, and $X_f$ to denote $X_f(\CC)$ in this subsection. The monodromy representation gives $$B_n\cong\pi_1(\cn, f)\to \mathrm{Aut}\big(H^1(X_f;\CC)\big)$$ with a choice of base point $f\in\cn$. 
The curve $X_f$ is a branch cover of $\CC$ by the projection map
\begin{align*}
X_f=\{(x,y):y^d=f(x)\}&\longrightarrow \CC\\
(x,y)&\mapsto x
\end{align*}
The branch points are exactly the $n$ roots of $f(x)$. The deck group is cyclic of order $d$ and generated by the map $T:(x,y)\mapsto (x,e^{2\pi i/d}y)$.  Since $T^d=\mathrm{Id}$, the eigenvalues of the induced linear map $T^*:H^1(X_f;\CC)\to H^1(X_f;\CC)$ are $d$-th roots of unity. We can decompose $H^1(X_f;\CC)$ into sum of eigenspaces of $T^*$. \begin{equation}
\label{eigen}
H^1(X_f;\CC)\cong \bigoplus_{\zeta:\zeta^d=1}H^1(X_f;\CC)_\zeta 
\end{equation}
where $H^1(X_f;\CC)_\zeta:=Ker(T^*-\zeta I)$.
Moreover, the monodromy action $B_n\curvearrowright H^1(X_f;\CC)$ commutes with $T^*$ and therefore preserves the eigenspaces. Hence (\ref{eigen}) is in fact a decomposition of $H^1(X_f;\CC)$ into $B_n$-subrepresentations.

\begin{lemma}
\label{monodromy}
For $\zeta$ such that $\zeta^d=1$, 
\begin{itemize}
\item If $\zeta=1$, then $H^1(X_f;\CC)_\zeta=0$.
\item If  $\zeta\ne1$, then $H^1(X_f;\CC)_\zeta\cong V_n(\zeta)^\vee $ as $B_n$-modules.
\end{itemize}
\end{lemma}
Recall $V_n(\zeta)^\vee$ is the dual of the reduced Burau representation specialized at $t=\zeta$, which was defined in (\ref{check}).
\begin{rmk}
In the beautiful paper \cite{McMullen}, McMullen studied the monodromy action of $B_n$ on the first cohomology of the smooth completion of $X_f$. Along the way, he proved several results in section 5 of \cite{McMullen} from which Lemma \ref{monodromy} can be deduced.  For completeness, we will prove Lemma \ref{monodromy} from a different perspective in the Section \ref{appendix}. 
\end{rmk}

\begin{prop}
\label{vanishing}
Suppose $n>2$ 
\begin{enumerate}
\item If either $n$ or $d$ is an odd number, then
$$H^p(\cn;H^1(X_f;\CC))=0\ \ \ \ \ \ \ \ \ \text{for all $p$}$$
\item If both $n$ and $d$ are even numbers, then 
\[
dim_\CC H^p(\cn;H^1(X_f;\CC))=
  \begin{cases}
   1 \ \ \ &  p=n-2,n-1 \\
   0 \ \ \ &  \text{otherwise} \\
  \end{cases}
\]
\end{enumerate}
\end{prop}
\begin{proof}
Since $\cn$ is a $K(B_n,1)$, we have $$H^p(\cn;H^1(X_f;\CC)) \cong H^p(B_n;H^1(X_f;\CC)).$$
By Lemma \ref{monodromy}, we can decompose
\begin{align*}
H^p(B_n;H^1(X_f;\CC))& \cong H^p(B_n;\bigoplus_{\zeta:\zeta^d=1}H^1(X_f;\CC)_\zeta)\\
& \cong H^p(B_n;\bigoplus_{\zeta^d=1,\zeta\ne1}V_n(\zeta)^\vee)\\
& \cong \bigoplus_{\zeta^d=1,\zeta\ne1}H^p(B_n;V_n(\zeta)^\vee)
\end{align*}
Each individual summand $H^p(B_n;V_n(\zeta)^\vee)$ has been computed in Proposition \ref{specialized Burau}. The claim then follows.
\end{proof}

\begin{prop}
\label{topo input}
When either $n$ or $d$ is odd, the bundle projection $\pi:\mathrm{E}_{n,d}(\CC)\to \nn(\CC)$  induces isomorphisms on the cohomology groups:
$$\pi^*:H^*(\Pn(\CC);\CC)\longrightarrow H^*(\mathrm{E}_{n,d}(\CC);\CC) $$
\end{prop}
\begin{proof}
Under the identification $\Pn(\CC)\cong\cn$, the fibration $X_f\to \mathrm{E}_{n,d}(\CC)\xrightarrow\pi \cn$ gives a spectral sequence of the following form:
$$E^{p,q}_2 = H^p(\cn;H^q(X_f;\CC)) \Longrightarrow H^{p+q}(E;\CC).$$
\begin{itemize}
\item When $q=0$ and $p>2$. $E^{p,0}_2 = H^p(\cn;\CC) = 0$  by the classical result of Arnol'd \cite{Arnol'd}. 
\item When $q=1$, 
$E^{p,1}_2 = H^p(\cn;H^1(X_f;\CC))=0$ by Proposition \ref{vanishing}.
\item When $q>1$, $E^{p,q}_2=0$ since $H^q(X_f;\CC)=0$ for $q>1$.
\end{itemize}
In summary, $E^{p,q}_2=0$ unless $q=0$, and hence all differentials are zero starting at the $E_2$ page. So $E_2 = E_\infty$ and there is no extension problem.
$$H^k(\mathrm{E}_{n,d}(\CC);\CC) \cong \bigoplus_{p+q=k}E_2^{p,q} = E^{k,0}_2 = H^k(\cn;\CC).$$
The isomorphism is induced by the fibration $\pi:\mathrm{E}_{n,d}(\CC)\to\cn$. 
\end{proof}

\subsection{From cohomology to counting points over $\fq$}
\label{etale}
In this section, we prove Corollary \ref{average pt} stated in the Introduction. The topological input will be Proposition \ref{topo input}. The bridge between the topology over $\CC$ and counting $\fq$-points is provided by the theory of \'etale cohomology.

\begin{proof}[Proof of Corollary \ref{average pt}]

First of all, notice that $$\sum_{f\in\nn{\fq}}|X_f(\fq)|=|\mathrm{E}_{n,d}(\fq)|.$$ Therefore, it suffices to prove that $\frac{|\mathrm{E}_{n,d}(\fq)|}{|\nn(\fq)|}=q$. 

For $p$ a prime number, we consider $\mathrm{E}_{n,d}$ and $\nn$ over the algebraically closed field $\overline{\ff}_p$. For $q$ a power of $p$, there is a \emph{Frobenius automorphism} $\mathrm{Frob}_q:z\to z^q$ on $ \overline{\ff}_p$, where the fixed points are exactly elements in $\fq$. Similarly, the (geometric) Frobenius automorphism acts on $\mathrm{E}_{n,d}(\overline{\ff}_p)$ and $\nn(\overline{\ff}_p)$ where fixed point sets are the following
\begin{align*}
\{\text{Fixed points of }\mathrm{Frob}_q\curvearrowright \mathrm{E}_{n,d}({\overline{\mathbb{F}}_p})\} &= \mathrm{E}_{n,d}{(\mathbb{F}}_q)\\
\{\text{Fixed points of }\mathrm{Frob}_q\curvearrowright {\nn}(\overline{\mathbb{F}}_p)\} &= {\nn}{(\mathbb{F}}_q)
\end{align*}

The Grothendieck-Lefschetz fixed point formula in our case gives
\begin{align}
\label{GL nn}|\nn| &= \text{Number of fixed points of }\mathrm{Frob}_q\curvearrowright \nn(\overline{\mathbb{F}}_p)\\
\nonumber
& = \sum_{i\geq0} (-1)^i \mathrm{Trace}(\mathrm{Frob}_q: H^i_{c}({\nn}_{/\overline{\mathbb{F}}_p};\QQ_{l})) \\ 
\label{GL E}|\mathrm{E}_{n,d}(\fq)| &= \text{Number of fixed points of }\mathrm{Frob}_q\curvearrowright \mathrm{E}_{n,d}(\overline{\mathbb{F}}_p)\\
\nonumber
& = \sum_{i\geq0} (-1)^i \mathrm{Trace}(\mathrm{Frob}_q: H^i_{c}(E_{{n,d/{\overline{\mathbb{F}}_p}}};\QQ_{l})) \ \ 
\end{align}
where $H^i_c$ stands for compactly supported \'etale cohomology. As usual, the coefficients are taken in $\QQ_l$ for $l$ relatively prime to $q$.

\begin{lemma}
\label{primes}
For all but finitely many $p$, the map $\pi:\mathrm{E}_{n,d}\to\nn$ induces an isomorphism 
\begin{equation}
\label{isom}
{\pi^*_{\text{\'et}}}_{{/\overline{\mathbb{F}}_p}}: H^i_{\text{\'et}}({\nn}_{/\overline{\mathbb{F}}_p};\QQ_l)\xrightarrow{\cong} H^i_{\text{\'et}}({\mathrm{E}_{n,d}}_{{/{\overline{\mathbb{F}}_p}}};\QQ_l)
\end{equation}
which is equivariant with respect to the action of $\mathrm{Frob}_q$, for all $i$.
\end{lemma}
\begin{proof}[Proof of the claim]
Since $\pi$ commutes with $\mathrm{Frob}_q$, the induced map $\pi^*_{\text{\'et}/{\overline{\mathbb{F}}_p}}$ must be equivariant with respect to $\mathrm{Frob}_q$. 

The lemma will follow from the commutative diagram
\begin{equation*}
\begin{tikzpicture}
\diagram (m)
{ H^i_\text{sing}({\nn}_{/\CC};\QQ_l)& H^i_\text{sing}({\mathrm{E}_{n,d}}_{/\CC};\QQ_l) \\
H^i_{\text{\'et}}({{\nn}_{/\CC}};\QQ_l) &H^i_{\text{\'et}}({{\mathrm{E}_{n,d}}_{/\CC}};\QQ_l) \\  
H^i_{\text{\'et}}({\nn}_{/\overline{\mathbb{F}}_p};\QQ_l)& H^i_{\text{\'et}}({{\mathrm{E}_{n,d}}_{/{\overline{\mathbb{F}}_p}}};\QQ_l) \\};
\path [->] (m-1-1) edge node [above] {$\pi^*$} (m-1-2)
           (m-1-2) edge node [right]  {$\cong$}(m-2-2)
           (m-1-1) edge node [left]  {$\cong$}(m-2-1)
           (m-2-1) edge node [above]{$\pi^*_{\text{\'et}/\CC}$}(m-2-2)
           (m-2-1) edge node [left]  {$\cong$}(m-3-1)
           (m-3-1) edge node [above]{$\pi^*_{\text{\'et}/{\overline{\mathbb{F}}_p} }$} (m-3-2)
           (m-2-2) edge node [right]{$\cong$} (m-3-2);

\end{tikzpicture}
\end{equation*}
where the top row is singular cohomology, the last two rows are \'etale cohomology of varieties over $\CC$ and ${\overline{\mathbb{F}}_p}$, respectively. The vertical maps are all isomorphisms. The horizontal maps are induced by $\pi$ in various categories. 
The commuting square on the top is given by the Artin's Comparison Theorem. The  commuting square on the bottom follows from base change for compactly supported \'etale cohomology and constructibility of the derived pushforward, which are respectively Theorems 5.4 and 6.2 in SGA $4\frac{1}{2}$ \cite{SGA}, Arcata, Chapter IV. We need to exclude a finite set of primes because the derived pushforward of the constant sheaf on $E_{n,d}$ to $\mathrm{Spec}\ZZ$ is constructible, but not necessarily locally constant. 
\end{proof}

Finally, we conclude the proof of Corollary \ref{average pt} in the following steps. The Poincar\'e duality for the \'etale cohomology can be found as Theorem 3.1 in SGA $4\frac{1}{2}$ \cite{SGA}, Arcata, Chapter VI.

\begin{align*}
|\nn(\fq)| & = \sum_{i\geq0} (-1)^i \mathrm{Trace}(\mathrm{Frob}_q: H^i_{c}({\nn}_{/\overline{\mathbb{F}}_p};\QQ_{l})) \ \ &(\ref{GL nn}) \\
&= q^{dim (\nn)}\sum_{i\geq0} (-1)^i \mathrm{Trace}(\mathrm{Frob}_q: H^i_{\text{\'et}}({\nn}_{/\overline{\mathbb{F}}_p};\QQ_{l})^\vee)&\text{Poincar\'e duality}\\
&=q^{n}\sum_{i\geq0} (-1)^i \mathrm{Trace}(\mathrm{Frob}_q: H^i_{\text{\'et}}({\nn}_{/\overline{\mathbb{F}}_p};\QQ_{l})^\vee)
\end{align*}

\begin{align*}
|\mathrm{E}_{n,d}(\fq)| &= \sum_{i\geq0} (-1)^i \mathrm{Trace}(\mathrm{Frob}_q: H^i_{c}({\mathrm{E}_{n,d}}_{{/{\overline{\mathbb{F}}_p}}};\QQ_{l}))&(\ref{GL E}) \\
&= q^{dim (\mathrm{E}_{n,d})}\sum_{i\geq0} (-1)^i \mathrm{Trace}(\mathrm{Frob}_q: H^i_{\text{\'et}}({\mathrm{E}_{n,d}}_{{/{\overline{\mathbb{F}}_p}}};\QQ_{l})^\vee)&\text{Poincar\'e duality}\\
&=q^{n+1}\sum_{i\geq0} (-1)^i \mathrm{Trace}(\mathrm{Frob}_q: H^i_{\text{\'et}}({\mathrm{E}_{n,d}}_{{/{\overline{\mathbb{F}}_p}}};\QQ_{l})^\vee)& \text{(\ref{isom}) is an isomorphism}\\
\end{align*}
Corollary \ref{average pt} follows by taking the ratio. Notice it is classically known that $|\nn(\fq)|=q^n-q^{n-1}\ne0$.
\end{proof}

\subsection{Proof of Lemma \ref{monodromy}}
\label{appendix}

\begin{proof}
The projection $p:X_f\to\CC$ is a branch cover. If we consider $D_n:= \CC\setminus\{\text{the $n$ branch points}\}$ and  $Y_f:= X_f\setminus \{\text{the $n$ ramification points}\}$, then $p|_{Y_f}: Y_f\to D_n$ is a (unramified) covering. The deck group is still $\ZZ/d\ZZ$, generated by $T$. Moreover, the deck transformation is given by composing the maps $ \pi_1(D_n)\xrightarrow\omega\langle t\rangle\cong\ZZ\to\ZZ/d\ZZ$, where $\omega$ is the total winding number defined in \ref{winding_def}. 

Recall $\CC_\zeta$ abbreviates
$\CC[t,t^{-1}]/(t-\zeta)$ as a $\CC[t,t^{-1}]$-module (see the first paragraph of Section \ref{specialized Burau section}). If $\zeta^d=1$, then $\CC_\zeta$ will be a representation of $\ZZ/d\ZZ$ where the generator $t\in\ZZ/d\ZZ$ acts by $\zeta$. The group ring $\CC[\ZZ/d\ZZ]$ decomposes into a direct sum of 1-dimensional irreducible representations $\CC_\zeta$
$$\CC[\ZZ/d\ZZ] = \bigoplus_{\zeta:\zeta^d=1}\CC_\zeta.$$

Therefore, we have 
\begin{align}
\nonumber H_1(Y_f;\CC)&=H_1(D_n;\CC[\ZZ/d\ZZ]) &\text{Shapiro's lemma}\\
\nonumber&=H_1(F_n;\CC[\ZZ/d\ZZ]) &\text{$D_n$ is a $K(F_n,1)$}\\
\nonumber&=H_1(F_n;\bigoplus_{\zeta:\zeta^d=1}\CC_\zeta) &\ \\
\label{decomp}&=\bigoplus_{\zeta:\zeta^d=1}H_1(F_n;\CC_\zeta) &\ 
\end{align}
Moreover, $T$ acts on $H_1(Y_f;\CC)$ preserving each summand in (\ref{decomp}) and acts on $H_1(F_n;\CC_\zeta)$ by $\zeta$. Hence, $H_1(Y_f;\CC)_\zeta = H_1(F_n;\CC_\zeta)$.

$X_f$ is obtained from $Y_f$ by attaching $n$ disks to cover the punctures. By the Mayer-Vietoris sequence, we see that the inclusion $Y_f\hookrightarrow X_f$ induces a surjective map on $H_1$ where the kernel is $H_1(Y_f;\CC)^{T}$, the $T$-invariant subspace of $H_1(Y_f;\CC)$. Another way to see this is to observe $H_1(Y_f;\CC)^{T}$ are spanned by cycles represented by small loops around the punctures on $Y_f$ because near the punctures, the deck transformation locally looks like a rotation by $2\pi/d$. Those loops are killed in $H_1$ when we cover the punctures by disks. In conclusion, we have 
\begin{equation}
\label{XY}
H_1(X_f;\CC)=H_1(Y_f;\CC)/H_1(Y_f;\CC)^{T}.
\end{equation}
As a result, $H_1(X_f;\CC)$ has no $T$-invariant vector, which is equivalent to saying $H_1(X_f;\CC)_\zeta=0$ for $\zeta=1$. 

When $\zeta\ne1$, we have $H_1(X_f;\CC)_\zeta\cong H_1(Y_f;\CC)_\zeta \cong H_1(F_n;\CC_\zeta)$ by (\ref{decomp}) and (\ref{XY}). We will show that in fact $H_1(F_n;\CC_\zeta)\cong V_n(\zeta)$ as $B_n$-modules. Notice 
\begin{itemize}
\item $V_n(\zeta) := V_n\underset{\RR}\otimes\CC_\zeta= H_1(F_n;\RR)\underset{\RR}\otimes\CC_\zeta$
\item $H_1(F_n;\CC_\zeta) := H_1(F_n;\RR\underset{\RR}\otimes\CC_\zeta)$
\end{itemize}
It follows from the universal coefficient theorem that $V_n(\zeta)=H_1(F_n;\CC_\zeta)$ because
$$Tor(H_0(F_n;\RR), \CC_\zeta) = Tor(\RR/(1-t),\CC_\zeta)=0$$ provided that $\zeta\ne1$. 

To sum up, we have proved  
\[
H_1(X_f;\CC)_\zeta = 
  \begin{cases}
   0 \ \ \ &  \text{if }\zeta=1 \\
   V_n(\zeta) \ \ \ &  \text{if }\zeta\ne1 
  \end{cases}
\]
The proof of Lemma \ref{monodromy} is completed by taking the dual $H^1(X_f;\CC) = H_1(X_f;\CC)^\vee$.
\end{proof}

\subsection{When both $d$ and $n$ are even}
\label{even}

When both $n$ and $d$ are even, our proof does not go through because the bundle projection $\pi:\mathrm{E}_{n,d}(\CC)\to\cn$ no longer induces isomorphisms on all the cohomology group in this case. In fact, by Proposition \ref{vanishing}, when both $n$ and $d$ are even we have
\begin{equation}
\label{E}
H^k(\mathrm{E}_{n,d};\CC)\cong
  \begin{cases}
   \CC \ \ \ &  k=0\\
   \CC \ \ \ &  k=1\\
   0 \ \ \ &  2<k<n-1 \\
   \CC \ \ \ & k=n-1\\
   \CC \ \ \ &  k=n\\
  \end{cases}
\end{equation}
Hence, $\mathrm{E}_{n,d}(\CC)$ and $\cn$ have non-isomorphic cohomology groups. 

After an early draft of this paper was distributed, Will Sawin, Frank Calegari, and Bjorn Poonen [personal communications] independently found proofs showing that when both $n$ and $d$ are even, 
\begin{equation}
\label{even formula}
\frac{1}{|\PP|}\sum_{f\in\PP} |X_f(\fq)| = q-q^{-(n-2)}.
\end{equation}
The cohomology computation (\ref{E}) provides a topological reason why an extra term of $-q^{-(n-2)}$ would appear in the expected number (\ref{even formula}) when both $n$ and $d$ are even. Since $|\PP|=q^n-q^{n-1}$, the equality (\ref{even formula}) is equivalent to 
$$|\mathrm{E}_{n,d}(\fq)|=q^{n+1}(1-q^{-1} -q^{-(n-1)}+q^{-n})$$
On the other hand, the Grothendieck-Lefschetz formula gives 
$$|\mathrm{E}_{n,d}(\fq)| = q^{n+1}\sum_{i\geq0} (-1)^i \mathrm{Trace}(\mathrm{Frob}_q: H^i_\text{\'et}({\mathrm{E}_{n,d}}_{{/{\overline{\mathbb{F}}_p}}};\QQ_{l})^\vee).$$
The  extra term of $-q^{-(n-2)}$ in the expected number (\ref{even formula}) comes from the contribution by $H^k(\mathrm{E}_{n,d};\CC)$ near the top dimensions $k=n-1,n$.

\section{The higher moments of the statistics}
\label{higher moments}

Fixing $d$ and $n$, if we view 
$|X_f(\fq)|$
as a random variable on $\nn(\fq)$, then Corollary \ref{average pt} computes the expected value of this random variable, when $d$ or $n$ is odd with a finite set of primes excluded. The next question is to determine the $m$-th moment of the random variable.

\begin{question}
\label{aquestion}
For each fixed $n$ and $m>1$,  what is
\begin{equation}
\label{moment}
\frac{1}{|\nn(\fq)|}\sum_{f\in\nn(\fq)}|X_f(\fq)|^m?
\end{equation}
\end{question}

On the topological side, Question \ref{aquestion} suggests us to consider the bundle over $\cn$ where the fiber over $f\in\cn$ is the $m$-th Cartesian product of curves $X_f\times \cdots\times X_f$, and to consider the following question:
\begin{question}
\label{tquestion}
Suppose $V_n^{\otimes m}$ is the $m$-th tensor power of the reduced Burau module, tensoring over $\CC$. When $m>1$, what is $H_k(B_n;V_n^{\otimes m})$ for each $k$ and each $n$?
\end{question}

Some partial answers to two questions are known. The asymptotic answer to Question \ref{aquestion} as $n\to \infty$ is worked out by Cheong-Wood-Zaman \cite{CWZ} and by Xiong \cite{Xiong}. In particular, each $m$-th moment (\ref{moment}) converges to a limit as $n\to\infty$. On the topological side, this leads one to wonder if $H_k(B_n;V_n^{\otimes m})$ stabilizes when $n\to \infty$. Indeed, it follows from Theorem 6.13 in \cite{Wahl} that for each fixed $k$ and $m$, the homology group $H_k(B_n;V_n^{\otimes m})$ is independent of $n$ when $n\geq 2k+m+2$.


\end{document}